\documentclass[12pt]{article}
\usepackage{multicol}
\usepackage[normalem]{ulem}
\usepackage{enumitem}
\usepackage{changepage}

\usepackage{pstricks}
\usepackage{pst-node}

\usepackage{amssymb}
\usepackage{amsthm}
\usepackage{tikz}
\usepackage{tkz-graph}
\usepackage{tkz-berge}
\usetikzlibrary{decorations.markings}
\usepackage{amsmath}
\usepackage{todonotes}
\usepackage{enumerate}
\usepackage{cite, comment}
\usepackage{color}
\usepackage{hyperref}
\usepackage{lineno}

\hypersetup{colorlinks=true, linkcolor=blue, citecolor=blue,urlcolor=blue}

\usepackage{graphicx}
\usepackage{placeins,caption}

\DeclareMathOperator{\sdi}{{\rm sd}_{\iota}}

\usepackage[top=3cm,bottom=3cm,left=2.8cm,right=2.8cm]{geometry}
	\newtheorem{theorem}{Theorem}[section]
	\newtheorem{lemma}[theorem]{Lemma}
	\newtheorem{corollary}[theorem]{Corollary}
	\newtheorem{proposition}[theorem]{Proposition}

	\newtheorem{obs}[theorem]{Observation}
	\theoremstyle{definition}
	
	\newtheorem{claim}[theorem]{Claim}

\usepackage{xargs}                      
\newcommandx{\unsure}[2][1=]{\todo[linecolor=red,backgroundcolor=red!25,bordercolor=red,#1]{#2}}
\newcommandx{\change}[2][1=]{\todo[linecolor=blue,backgroundcolor=blue!25,bordercolor=blue,#1]{#2}}
\newcommandx{\PZ}[2][1=]{\todo[linecolor=OliveGreen,backgroundcolor=OliveGreen!25,bordercolor=OliveGreen,#1]{#2}}
\newcommandx{\MMcos}[2][1=]{\todo[linecolor=blue,backgroundcolor=blue!25,bordercolor=blue,#1]{#2}}
\newcommandx{\improvement}[2][1=]{\todo[linecolor=Plum,backgroundcolor=Plum!25,bordercolor=Plum,#1]{#2}}
\newcommandx{\thiswillnotshow}[2][1=]{\todo[disable,#1]{#2}}

\title{Isolation subdivision number of a graph}
\author{}
\date{}

\begin{document}

\maketitle

\vspace{-10mm}\noindent \textbf{Magda Dettlaff}, University of Gdańsk, 80-309 Poland\\ \textbf{Magdalena Lema\'nska}, Gdańsk University of Technology, 80-233 Gdańsk, Poland\\ \textbf{Merce Mora}, Universitat Politècnica de Catalunya, 08-034 Barcelona, Spain \\ \textbf{Rados\l{}aw Ziemann}, University of Gdańsk, 80-309 Gdańsk, Poland\\ \textbf{Pawe\l{}  \.Zyli\'nski}, University of Gdańsk, 80-309 Gdańsk, Poland

\begin{abstract}
\noindent For a graph $G=(V,E),$ a set $S \subseteq V$ is called an isolating set of $G$ if  the set $V-N[S]$ is independent. The minimum cardinality of an isolating set in $G$ is the  isolation number of $G$, denoted by $\iota(G).$  
Here we introduce  the isolation subdivision number of a graph $G$, denoted by $\sdi(G)$, as the minimum number of edges of $G$ that must be subdivided, where each edge can be subdivided at most once, in order to obtain a graph with isolation number greater than $\iota(G).$ 
We show that the new parameter is well defined for any non-trivial graph different from a star and that it can be arbitrarily large. We present the values of this parameter for some elementary classes of graphs and establish some basic properties. 
We show also that $1\leq \sdi(T)\leq 4$ for any tree $T$ different from a star  and characterize all trees $T$ with $\sdi(T)=1.$ \\[2mm]
\end{abstract}

\section{Introduction}

A set $D$ of vertices of a graph $G=(V(G),E(G))$ is {\it dominating} if every vertex not in $D$ has at least one neighbour in $D$ and the {\it domination number} of $G$ is a cardinality of a minimum dominating set of $G$.  
Clearly, if $N_G[D]$ denotes the set of vertices of $D$ and all their neighbours,
then $D$ is a dominating set of $G$  if and only if  $V(G)=N_G[D]$. 
The concept of {\it isolation} arises by relaxing the condition  $V(G)=N_G[D]$. 
It was introduced by Caro and Hansberg in \cite{adriana}, and then extensively studied in~\cite{Borg1,Borg2,Borg3, Borgrsc,BFK,BFK2,BK23,BG24,BGH24,ChLWX25,ChCZ25,FK21,HAV25,LMSS24}, from different perspectives, to mention just a few. 
In particular, for a family $\mathcal{F}$  of graphs, a set $D$ of vertices of a graph $G$ is $\mathcal{F}$-{\it isolating}  if the graph induced by the set $V(G)- N_G[S]$ contains no copy of a member of $\mathcal{F}$ as a subgraph. 
 The {\it $\mathcal{F}$-isolation number} of $G$, denoted by  $\iota(G, \mathcal{F})$, is the size of a smallest $\mathcal{F}$-isolating set of $G$. 
 For example,  $\{K_1\}$-isolating sets are the usual  dominating sets and
the vertices not dominated by a $ \{ K_2 \}$-isolating set form an independent set.
If $\mathcal{F} = \{K_2\}$, then the terms $\mathcal{F}$-isolating set and $\mathcal{F}$-isolation number, and the notation $\iota(G,\mathcal{F})$, are abbreviated to isolating set, isolation number, and $\iota(G)$, respectively. 

Isolation in graphs was introduced as an extension of domination, 
by removing the constraint that every vertex must be dominated.
One can consider a potential application of this concept in the context of a communication network.
Suppose a graph $G$ represents a communication network and a security agency wants to detect every conversation between two members, i.e., two adjacent vertices, in this network. Then, the elements of a vertex subset $D$ corresponds to the certain centers, which can listen to all nodes in $N_G[D].$  As long as $V(G) \setminus N_G[D]$ is an independent set, the agency can still listen to all conversations by having microphones located in all vertices of set $D.$
We also note that isolation was introduced as {\it vertex-edge domination} ({\it ve-domination}) in \cite{peters} and also studied under this name, independently, in several papers \cite{BChHH15,KMY21,KVK14,L07,LHHF10,sentUVED,ZZ20,ZZ25,pawel}.

\paragraph{Problem statement.} We consider the concept of  isolation in the context of edge subdivision and study the effect of this operation on the isolation number. Namely, the {\it isolation subdivision number}, $\sdi(G)$, of a graph $G$ is the minimum number of  edges that must be subdivided (where each edge can be subdivided at most once) in order to obtain a graph with isolation number greater than $\iota(G)$. 
Since the isolation number of a star $K_{1,n}$ ($n\geq 1$) does not increase when  all its edges are subdivided, we therefore consider only connected non-trivial graphs different from stars. 

The idea of measure how operations (in particular, the edge subdivision) affect domination and related parameters is well 
known in the literature (see, for example, \cite{atapour,sub1, favaron, sub2, sub3, michael}). 
For a graph $G$, the {\it subdivision} of an edge $e=uv$ with a new vertex $x$ (called the {\it subdivision vertex}) is an operation which leads to a graph~$G_e$ with $V(G_{e})=V(G)\cup \{w\}$ and $E(G_{e})=(E(G)\setminus\{uv\})\cup \{ux,xv\}$. 
For $S\subseteq E(G)$,  
we denote by $G_S$ the graph obtained by subdividing all the edges in $S$.
The {\it domination subdivision number} of a graph $G$, denoted $sd_{\gamma}(G)$, 
is the minimum cardinality of a set $S$ of edges of $G$ such that the domination number of $G_S$ is greater
 than the domination number of $G$. 
 This concept for domination was defined by Arumugam 
in a private communication and studied in \cite{sub1,sub2} and \cite{sub3}.
Since this parameter was defined for classical domination, there have been many works defining it for different types of domination like for total domination \cite{michael}, Roman domination \cite{atapour} or paired domination \cite{favaron}.

\paragraph{Organization of the paper.} The work is organized as follows. We start with useful notation. Then we show that the parameter $\sdi(G)$ is well-defined.
We present the values of this parameter for some elementary classes of graphs and give some of its basic properties. Next, we concentrate on the isolation subdivision number of trees. In particular, w show that for any non-trivial tree $T$  different from a star we have $1\leq \sdi(T)\leq 4$. We also characterize all trees $T$ with $\sdi(T)=1.$ 

\paragraph{Notation.}
For notation and graph theory terminology we, in general, follow~\cite{HaHeHe-20,HaHeHe-21,HaHeHe-23}. Specifically, let $G=(V(G),E(G))$ be a connected graph and $u,v\in V(G)$.
For a vertex $v$ of $G$, its {\it neighborhood\/}, denoted by $N_{G}(v)$, is the set of all vertices adjacent to $v$, and the cardinality of $N_G(v)$, denoted by $\deg_G(v)$, is called the {\it degree} of~$v$. 
The {\it closed neighborhood\/} of $v$, denoted by $N_{G}[v]$, is the set $N_{G}(v)\cup \{v\}$. 
In general, for a subset $X\subseteq V(G)$, the {\it neighborhood\/} of $X$, denoted by $N_{G}(X)$, is defined to be $\bigcup_{v\in X}N_{G}(v)$, and the {\it closed\/} neighborhood of $X$, denoted by $N_{G}[X]$, is the set $N_{G}(X)\cup X$. 
A vertex of degree one in $G$ is called a {\em leaf\/} of $G$. A vertex adjacent to a leaf in $G$ is called a {\it support vertex} of $G$ and if the support vertex is adjacent to exactly one leaf $v$, we call $v$  a {\it weak leaf} of $G$ (otherwise $v$ is a {\it strong leaf} of $G$).

The {\it distance} between $u$ and $v$ in $G$, denoted by $d_G(u,v)$, is the length of a shortest path from $u$ to $v$ in $G$.  The {\it diameter} of a graph $G$, written as ${\rm diam}(G)$, is the maximum distance between any two vertices of $G$. By $d_G(v,X)=\min\{d_G(v,x)\colon x\in X\}$ let us denote the distance of a vertex $v$ from the set $X\subseteq V(G)$. 
We say that a vertex $v\in V(G)$ {\it ve-dominates an edge} $e=xy\in E(G)$ if and only if $N_G[v]\cap \{x,y\}\neq \emptyset$, i.e. $d_G(v,\{x,y\})\leq 1$. Hence, the set $D\subseteq V(G)$ is an isolating set of $G$ if each edge of $G$ is ve-dominated by a vertex of $D$. 
Let $S\subseteq V(G)$, $v\in S$ and assume $e=xy\in E(G)$ is ve-dominated by $v$. 
We say that $e$ is a {\it private edge} for $v$ with respect to $S$ if and only if $v$ is the only vertex of $S$ that ve-dominates $e$, i.e., $N_G[z]\cap \{x,y\}=\emptyset$ for every $z\in S\setminus \{v\}$. 

Next, if $v$ is a vertex of $G$ and $e=xy$ is an edge of $G$, we define:
\begin{itemize}
\item $N_G[e]$ -- the set $N_G(x) \cup N_G(y)$;
\item $N^2_G(v)$ --  the set of vertices at distance two from $v$ vertices in $G$, i.e. $N^2_G(v)=\{u\in V(G)\colon d_G(u,v)=2\}$;
\item $E^1_G(v)$ -- the set of edges incident to $v$ in $G$;
\item $E^{11}_G(v)$ -- the set of edges with both endpoints in $N_G(v)$;
\item $E^{2}_G(v)$ -- the set of edges with one endpoint in $N_G(v)$ and the other in  $N^2_G(v)$;
\item $\mathcal{E}(v)=E^1_G(v) \cup E^{11}_G(v) \cup E^2_G(v)$;
\item $X_{\mathcal{E}(v)}$ -- the set of subdivision vertices of $G_{\mathcal{E}(v)}$.
\end{itemize}

Now, let $\prec_{G_{\mathcal{E}(v)}}$ be some vertex ordering of  the vertices in~$G_{\mathcal{E}(v)}$, with $v$ as the smallest element. For a vertex $u \in V(G_{\mathcal{E}(v)})$, we define the {\it $v$-shift} -- denoted by ${\rm sh}_{v}(u)$ -- {\it of $u$ with respect to ordering $\prec_{G_{\mathcal{E}(v)}}$} as follows:  
\begin{itemize}
    \item If $u \in V(G)$, we set ${\rm sh}_{v}(u)=u$. 
    \item If $u \in X_{\mathcal{E}(v)}$, then ${\rm sh}_{v}(u)$ is the smallest vertex 
    in $N_{G_{\mathcal{E}(v)}}(u)$ 
    with respect to the ordering $\prec_{G_{\mathcal{E}(v)}}$.
\end{itemize}
For a subset $S \subseteq V(G_{\mathcal{E}(v)})$, we define the {\it $v$-shift ${\rm sh}_v(S)$ of $S$ with respect to $\prec_{G_{\mathcal{E}(v)}}$} to be the set 
$\{{\rm sh}_v(u): u \in S\}$. 
Directly from the definition of the $v$-shift we obtain the following observations.
\begin{obs}\label{obs-vshift-1}
For any subset $S \subseteq V(G_{\mathcal{E}(v)})$, we have ${\rm sh}_v(S) \subseteq V(G)$ and $|{\rm sh}_v(S)| \le |S|$.     
\end{obs}

\begin{obs}\label{obs-vshift-2}
If $S$ is an isolating set of $G_{\mathcal{E}(v)}$, then ${\rm sh}_v(S)$ is an isolating set of $G$.     
\end{obs}

Finally, let $L_G$ and $S_G$ denote the set of leaves and the set of supporting vertices, respectively, in $G$. And, for a positive integer $k$, let $[k]$ denote the set $\{1,2,\ldots,k\}$.

\section{Preliminaries}
In this section we prove that the isolation subdivision number is well defined for any non-trivial graph $G$ different from a star, which means that we can always find a set of edges $F\subseteq E(G)$ such that $\iota(G_F)>\iota(G)$. Later, we provide some general  observations concerning the structural properties of graphs. Finally, we show that the isolation subdivision number can be arbitrarily large.

First, observe that subdividing an edge never decreases the isolation number of a graph and increases it by at most one.
\begin{proposition}
If $G$ is a graph, then $$\iota(G)\leq \iota(G_e)\leq \iota(G)+1.$$
\end{proposition}
\begin{proof}
Let $G_e$ be a graph obtained from $G$ by subdividing an edge $e=uv$ by a vertex $x$. Let $D$ be a minimum isolating set of $G_e$. If $x\in D$, then $D'=(D\setminus \{x\})\cup \{u\}$ is an isolating set of $G$. Otherwise, if $x \notin D$, then $D'=D$ is an isolating set of $G$. Hence, in both cases $\iota(G)\leq |D'|\leq |D|=\iota(G_e).$ Now, let $D$ be a minimum isolating set of $G$. In this case $D'=D\cup \{x\}$ is an isolating set of $G_e$, which implies $\iota(G_e)\leq |D'|=|D|+1=\iota(G)+1$.
\end{proof}

\paragraph{The parameter ${\rm sd}(\cdot )$ is well defined.}
We are now in a position to show that for any graph $G$, there exists a set of edges such that subdividing all its edges yields a new graph with isolation number greater than $\iota(G)$.
We begin our analysis with a useful lemma.
\begin{lemma}\label{lem-vshift-1}
If $\iota(G_{\mathcal{E}(v)})=\iota(G)$ then there is no $\iota$-set $D_v$ of $G_{\mathcal{E}(v)}$ such that $|N_{G_{\mathcal{E}(v)}}[v] \cap D_v| \ge 2$.     
\end{lemma}
\begin{proof}
Suppose to the contrary that there exists an $\iota$-set $D_v$ of $G_{\mathcal{E}(v)}$ such that $|N_{G_{\mathcal{E}(v)}}[v] \cap D_v| \ge 2$ for some $v \in V(G)$. 
Then we have $|{\rm sh}_v(D_v)|< |D_v|$, and  ${\rm sh}_v(D_v)$ is an isolating set of $G$ by Observation~\ref{obs-vshift-2}, a contradiction with $\iota(G_{\mathcal{E}(v)})=\iota(G)$. 
\end{proof}

\begin{theorem}
    If $G$ is a graph different from a star, then $\sdi(G)$ is well defined.
    \label{welldefined}
\end{theorem}

\begin{proof} First assume that each vertex of $G$ is either a leaf or a support. Observe that for a graph $G=(V(G),E(G))$, if $V(G)=L_G \cup S_G$ then $\iota(G) \le \gamma(G[S_G])$. On the other hand, if $\bar{G}$ is the graph resulting from subdividing one pendant edge per each relevant supporting vertex in $S_G$, then we have $\iota(\bar{G})=|S_G|$. Consequently, if $\iota(\bar{G}))=\iota(G)$, then $|S_G| \le \gamma(G[S_G])$, thus implying $|S_G|=1$, which immediately results in $G$ being a star. Consequently, if $G$ is not a star, then we must have $\sdi(G) \le |S_G| \le |V(G)|/2$.
    
    Now, assume there is a non-leaf non-support vertex in $G$, let $v \in V(G)$ be such a vertex. Assume $N_{G}(v)=\{v_1,v_2, \ldots, v_{p}\}$, and so $E^1_G(v_1) =\{vv_1,vv_2,\ldots,vv_p\}$, while each edge in $E^{11}_G(v)$ -- if non-empty -- is of the form $v_iv_j$, where $i,j \in [p]$, $i \neq j$. Next, assume that $N^2_G(v)=\{v_{p+1},v_{p+2},\ldots,v_t\}$, and so each edge in $E^{2}_G(v)$ -- if non-empty -- is of the form $v_iv_q$, where $i \in [p]$ and $q \in [t] \setminus [p]$. 

Consider the graph $G_{\mathcal{E}(v)}$. Assume that:
\begin{itemize}
\item each edge $vv_i \in E^1(v_1)$, where $i \in [p]$, is subdivided by $x_i$;
\item each edge $v_iv_j \in E^{11}_G(v)$, if any, where $i,j \in [p]$, $i \neq j$, is subdivided by $y_{i,j}$;
\item each edge $v_iv_q \in E^2_G(v_1)$, if any, where $i \in [p]$ and $q \in [t] \setminus [p]$, is subdivided by  $z_{i,j}$.
\end{itemize}

Suppose now to the contrary that $\iota(G_{\mathcal{E}(v)})=\iota(G)$ and let $D^\circ$ be an $\iota$-set of $G_{\mathcal{E}(v)}$. To ve-dominate the edge $vx_1$, we must have $D^\circ \cap (N_{G_{\mathcal{E}(v)}}[v] \cup \{v_1\} \neq \emptyset$. 
\\[2mm]
{\it Case $1$}: $v \in D^\circ$. We consider the ordering $\prec_{G_{\mathcal{E}(v)}} = ( v,v_1,v_2,v_3, \ldots, v_t, \ldots)$. Since $v$ is not a support vertex in $G$, we must have $\deg_{G_{\mathcal{E}(v)}}(v_i) \ge 2$ for any vertex $v_i$, $i \in [p]$. Consider now the vertex $v_i$, for each $i \in [p]$, and any edge $e_i \in E^1_{G_{\mathcal{E}(v)}}(v_i) \setminus \{x_iv_i\}$. To ve-dominate the edge $e_i$, keeping in mind Lemma~\ref{lem-vshift-1}, 
we must have $$D^\circ \cap \left(N_{G_{\mathcal{E}(v)}}[e_i] \setminus \{x_i\}\right) \neq \emptyset,$$
say $u_{v_i} \in D^\circ \cap \left(N_{G_{\mathcal{E}(v)}}[e_i] \setminus \{x_{i}\}\right)$, which immediately implies ${\rm sh}_{v}(u_{v_i}) \in N_G[v_i] \setminus \{v\}$ for each $i \in [p]$. Consequently, the set ${\rm sh}_{v}({D}^{\circ}) \setminus \{v\}$ is an isolating set of $G$, a contradiction with $\iota(G_{\mathcal{E}(v)})=\iota(G)$. 
\\[2mm]
{\it Case $2$}: $v \notin D^\circ$ and $|D^\circ \cap N_{G_{\mathcal{E}(v)}}(v)| =1$ (see Lemma~\ref{lem-vshift-1}). Without loss of generality assume $x_1 \in D^\circ$. 
We consider two subcases.
\\[2mm]
{\it Subcase $2.1$}: $N^2_G(v) \cap N_G(v_1) \neq \emptyset$; we consider the ordering $\prec_{G_v} = (v, v_1,v_2,v_3, \ldots, v_t, \ldots )$, notice ${\rm sh}_v(x_1)=v$. Then, there exists an edge $v_1v_{q} \in E(G)$ for some $q \in [t] \setminus [p]$, and so the edge $e_1=z_{1,q}v_{q} \in E_{G_{\mathcal{E}(v)}}$. To ve-dominate the edge $e_1$, we must have $D^\circ \cap N_{G_{\mathcal{E}(v)}}[e_1] \neq \emptyset,$
say $u_{v_1} \in D^\circ \cap N_{G_{\mathcal{E}(v)}}[e_1]$, which  implies ${\rm sh}_{v}(u_{v_1}) \in N_G[v_q]$. Next, for each vertex $v_j$, where $j \in [p] \setminus \{1\}$, keeping in mind $|D^\circ \cap N_{G_{\mathcal{E}(v)}}(v)| =1$ and ve-domination of edge $x_jv_j$ by $D^\circ$, one can argue the existence of the relevant vertex $u_{v_j}$ such that ${\rm sh}_{v}(u_{v_j}) \in N_G[v_j] \cap N^1_G(v)$, with ${\rm sh}_v(u_{v_j}) \neq v$. Consequently, since $\deg_G(v) \ge 2$, 
the set ${\rm sh}_{v}(D^\circ) \setminus \{v\}$ is an isolating set of $G$, a contradiction with $\iota(\bar{G})=\iota(G)$.
\\[2mm]
{\it Subcase $2.2$}: $N^2_G(v) \cap N_G(v_1) = \emptyset$. Since $v$ is not a support in $G$, there exists $j \in [p] \setminus \{1\}$ such that edge $v_1v_j \in E(G)$. Now, to ve-dominate the edge $e_j=x_jv_j$, keeping in mind Lemma~\ref{lem-vshift-1}, we must have $D^\circ \cap (N_{G_{\mathcal{E}(v)}}[e_j] \setminus \{v,x_{j}\}) \neq \emptyset$, say $u_j \in D^\circ \cap (N_{G_{\mathcal{E}(v)}}[e_j] \setminus \{v,x_{j}\})$, which now, by setting the ordering $\prec_G = (v,v_j,v_1,v_2,v_3, \ldots, v_{j-1},v_{j+1}, \ldots,v_p,\ldots)$, implies ${\rm sh}_{v}(u_j)=v_j$. Next, 
for each vertex $v_k$, where $k \in [p] \setminus \{1,j\}$, one can also argue the existence of the relevant vertex $u_{v_k}$ such that ${\rm sh}_{v}(u_{v_k}) \in N_G[v_k] \setminus \{v\}$ (recall $|D^\circ \cap N_{G_{\mathcal{E}(v)}}(v)| =1$). Consequently,
the set ${\rm sh}_{v}(D^\circ) \setminus \{v\}$ is an isolating set of $G$ (we emphasize that $v_j \in {\rm sh}_{v}(D^\circ)$), a contradiction with $\iota(G_{\mathcal{E}(v)})=\iota(G)$. 
\\[2mm]
{\it Case $3$}: $D^\circ \cap N_{G_{\mathcal{E}(v)}}[v] = \emptyset$. Then, to ve-dominate all edges $\{vx_{i}\} \in E_{G_{\mathcal{E}(v)}}$, we must have  $v_i \in D^\circ$ for each $i \in [p]$. But then, the set $\big({\rm sh}_{v}(D^\circ) \setminus \{v_1,v_2,v_3,\ldots,v_p\}\big) \cup \{v\}$ is an isolating set of $G$ (since $\deg_G(v) \ge 2$), a final contradiction with $\iota(G_{\mathcal{E}(v)})=\iota(G)$. 
\end{proof}

By Theorem~\ref{welldefined}, the parameter $\sdi(G)$ is well-defined for any graph $G$ with no component isomorphic to a star. In particular we have the following.

\begin{corollary}
    Let $G=(V(G),E(G))$ is a graph with no component isomorphic to a star.
    \begin{itemize}
    \item[$a)$] If $V(G)=S_G \cup L_G$, then $\sdi(G) \le |V(G)|/2$.
    \item[$b)$] If $V(G) \neq S_G \cup L_G$, then 
    $$\sdi(G)\le \min_{v \in V(G) \setminus (S_G \cup L_G)}  |\mathcal{E}(v)| \le \min_{v \in V(G) \setminus (S_G \cup L_G)} \sum_{u \in N_G(v)} \deg_G(u).$$
    
    \end{itemize}
\end{corollary}

\paragraph{Some general observations.} 
The following results allow us to focus exclusively on graphs without strong supports.

\begin{lemma}\label{lem_support}
    If $G^*$ is a graph obtained from $G$ by removing all but one leaves for every support of $G$, then $\sdi(G^*)=\sdi(G)$.
\end{lemma}
\begin{proof}
    Let $u$ be a strong support adjacent to $k\geq 2$ leaves $v_1,v_2,\ldots v_k$. Then $\iota(G_{uv_1})=\iota(G_F)$, where $F$ is any subset of pendant edges $uv_i$ for $i\in [k]$.
\end{proof}

By Lemma~\ref{lem_support}, from now, we will consider only graphs without strong supports.

\begin{lemma}\label{lem_isd1}
    If $G$ has an induced path $P_5=(v_1,v_2,v_3,v_4,v_5)$ such that $\deg_G(v_1)=\deg_G(v_5)=1$ and $\deg_G(v_2)=\deg_G(v_4)=2$, then $\sdi(G)=1$.
\end{lemma}
\begin{proof}
    Let us consider a graph $G'$ obtained from $G$ by subdividing the edge $v_1v_2$ with vertex $x$. Let $D'$ be a minimum isolating set of $G'$. It is easy to observe $|D'\cap \{v_1,v_2,v_3,v_4,v_5,x\}|\geq 2$ and without loss of generality we can assume $\{v_2,v_3\}\subseteq D'$. On the other hand, any minimum isolating set $D$ of $G$ contains only $v_3$ from this path. This implies that $D'\setminus \{v_2\}$ is an isolating set of $G$ and $D\cup \{v_2\}$ is an isolating set of $G'$. Hence, $\iota(G)\leq \iota(G')-1$ and $\iota(G')\leq \iota(G)+1$, respectively, 
    which implies $\iota(G')=\iota(G)+1$, and thus $\sdi(G)=1$.
\end{proof}

\begin{obs}\label{o_1}
    For any graph, it is always possible to choose a minimum isolating set containing neither leaves nor support vertices for which all but one of the neighbors are leaves.
\end{obs}
We conclude this part by determining the subdivision isolation number of paths and cycles.
 In~\cite{adriana} was shown that for a path $\iota(P_n)=\lceil \frac{n-1}{4} \rceil$ and  for a cycle $\iota(C_n)=\lceil \frac{n}{4} \rceil$. Simple calculations lead to the following results. 
\begin{obs}
    For a path $P_n$ and for a cycle $C_n$ we have
    
    $\sdi(P_n)=\begin{cases}
1,\quad n\equiv 1\pmod 4\\
2, \quad n\equiv 0\pmod 4\\
3, \quad n\equiv 3\pmod 4\\
4, \quad n\equiv 2\pmod 4
        \end{cases}\ \  and\quad
   \sdi(C_n)=\begin{cases}
1,\quad n\equiv 0\pmod 4\\
2, \quad n\equiv 3\pmod 4\\
3, \quad n\equiv 2\pmod 4\\
4, \quad n\equiv 1\pmod 4.
        \end{cases}
   $
\end{obs}

\paragraph{The parameter $\sdi$ can be arbitrarily large.}
The aim of this paragraph is to show that, unlike for paths and cycles, for any positive integer $k$ there exists a graph $G$ such that $\sdi(G)=k$. To this end, we consider the isolation subdivision number of complete graphs.

\begin{theorem}

    If $G$ is a complete graph of order $n$, then $$\sdi(K_n)= \bigg\lceil \frac{2n}{3} \bigg\rceil.$$
\end{theorem}
\begin{proof}
    Let $C=(v_1,v_2,\ldots ,v_n)$ be a hamiltonian cycle of $K_n$. Let $K_n^*$ be the graph obtained from $G$ by subdividing the edges $v_iv_{i+1}$ with the vertices $x_i$ for each $i\neq 0\pmod 3$. Notice that if $i\equiv 2\pmod 3$, then $v_i$ does not ve-dominate the edge $v_{i-1}v_{i+1}$. If $i\equiv 1 \pmod 3$, then a vertex $v_i$ does not ve-dominate the edge $v_{i+1}x_{i+1}$, while if $i\equiv 0 \pmod 3$, then $v_i$ does not ve-dominate the edge $v_{i-1}x_{i-2}$. All indices are taken modulo 3.    
    Hence, $\iota (K_n^*)>1=\iota(K_n).$ Thus, $\sdi(K_n)\leq \big\lceil \frac{2n}{3} \big\rceil.$

    To prove that $\big\lceil \frac{2n}{3} \big\rceil$ is a lower bound, let $K_n^*$ be a graph obtained from $K_n$ by subdividing some set of edges $F$ of $K_n$ and let associate  with $K_n^*$ a graph $H$ in the following way: $V(H)=V(K_n)$ and $E(H)=F$.
   \begin{claim} \label{c_1}
       If $H$ has a component of order 1, then $\iota(K_n^*)=1.$
   \end{claim}
   \begin{proof}
    Let $\mathcal{C}$  be a component of order 1 of $H$ induced by $\{u\}$. In this case, $u$ is adjacent to every vertex of $K_n- u$ and hence every edge of $K_n^*$ is ve-dominated by $u$ since every edge of $K_n^*$ has at least one  endpoint in $V(K_n)$. Thus, $\iota(K_n^*)=1.$
   \end{proof}
    \begin{claim} \label{c_2}
       If $H$ has a component of order 2, then $\iota(K_n^*)=1.$
   \end{claim}
   \begin{proof}
    Let $\mathcal{C}$ be a component of order 2 of $H$ induced by $\{u,v\}$. Hence, $u$ is adjacent in $K_n^*$ to all the vertices in $V(K_n)\setminus \{u,v\}$, what implies that every edge of $K_n^*$ having at least one endpoint in $V(K_n)\setminus \{v\}$ is ve-dominated by $\{u\}$. The only remaining edge is $xv$ where $x$ is a subdivision vertex of $uv$. It is easy to observe that this edge of $K_n^*$ is also ve-dominated by $\{u\}$. Thus, $\iota(K_n^*)=1.$
   \end{proof}
   Let $C_1, C_2, \ldots, C_h$ be the components of $H.$
   If $\iota(K_n^*)\geq 2$, then by Claim~\ref{c_1} and Claim~\ref{c_2} we obtain that each component of $H$ has at least 3 vertices. Then $h \leq \big\lfloor \frac{n}{3}\big\rfloor.$ Therefore $|E(H)|= \sum_{i=1}^h |E(C_i)| \geq \sum_{i=1}^h (|V(C_i)|-1)=n-h \geq n - \big\lfloor \frac{n}{3}\big\rfloor=\big\lceil \frac{2n}{3}\big\rceil.$ Thus, $\sdi(K_n)\geq \big\lceil \frac{2n}{3} \big\rceil$, which completes the proof.
\end{proof}

\section{Isolation subdivision number of trees}
In this section, we investigate the isolation subdivision number for trees and prove that it is restricted to values within the set $\{1,2,3,4\}$.
\begin{theorem}
    If $T$ is a tree different from a star, then $1\leq \sdi(T)\leq 4$.
\end{theorem}
\begin{proof} If ${\rm diam}(T)\leq 4$ it is easy to check that the result holds. 
    Let $T$ be a tree with ${\rm diam}(T)\geq 5.$ Let $P=(v_0, v_1, \ldots, v_{\ell})$ be a longest path of $T$ such that $\deg_T(v_2)$ is as large as possible and let $D$ be a minimum isolating set of $T.$ 
    Our assumption implies that $\deg_T(v_1)=\deg_T(v_{\ell-1})=2.$ If $v_2$ is adjacent to a support vertex different from $v_1$ and $v_3$ (if $v_3 \in S_T$) then  $\sdi(T)=1$  by Lemma~\ref{lem_isd1}. Hence the only support adjacent to $v_2$  is $v_1$ and $v_3$ (provided $v_3 \in S_T$). Then we have two possibilities:  either $v_2$ is a support vertex or $\deg_T(v_2)=2.$

    Assume first $v_2$ is a support adjacent to a leaf $w.$ Then $\deg(v_2)=3.$  We subdivide the edges $v_0v_1, v_1v_2$ and $v_2w$ with vertices $x_1, x_2, x_3,$ respectively, obtaining a tree $T'.$ Since $v_0, w$ are leaves in $T',$ by Observation \ref{o_1}, there is a minimum isolating set $D'$ of $T'$ such that $\{v_1, v_2\} \subseteq D'.$ This implies that $D' \setminus \{v_1\}$ is an isolating set of $T$ and $\iota(T) \leq |D'|-1 =\iota(T')-1$ and we are done. So $\deg_T(v_2)=2$ (and by symmetry, $\deg_T(v_{\ell-2})=2$).

Assume now there is a leaf $w$ at distance two from $v_3$ and let $z$ be a (weak) support adjacent to $w.$  We subdivide the edges $v_0v_1, v_1v_2$, $v_2v_3$ and $wz$ with vertices $x_1, x_2, x_3,x_4$ respectively, obtaining the tree $T'.$ Let $D'$ be a minimum isolating set of $T'$ such that (by Observation \ref{o_1}) $\{v_1,z\} \subseteq D'$. Observe that $|\{v_2,x_3,v_3\} \cap D'|\geq 1$ and without loss of generality let $v_3 \in D'.$ In this case $D' \setminus \{z\}$ is an isolating set of $T$ of cardinality $\iota(T')-1.$ Thus $\iota(T) \leq \iota(T')-1$ and we are done. 

Now we can assume that $v_3$ has no distance-2-leaves; so all leaves are at distance three or one from $v_3$. Consider the case when $v_3$ is a support vertex adjacent to a leaf $w$ (which implies $\ell\geq 6$). 
We subdivide the edges $v_0v_1, v_1v_2$, $v_3v_4$ and $v_3w$ with vertices $x_1, x_2, x_3,x_4$ respectively, obtaining the tree $T'.$ Let $D'$ be a minimum isolating set of $T'$ such that (by Observation \ref{o_1}) $\{v_1,v_3\} \subseteq D'$. Observe that we can choose $D'$ in such a way that $x_3 \notin D'.$ In this case $(D' \setminus \{v_1,v_3\})\cup \{v_2\}$ is an isolating set of $T$ of cardinality $\iota(T')-1.$ Thus $\iota(T) \leq \iota(T')-1$ and we are done. 

Now we can assume that $v_3$ is not a support and has no distance-2-leaves. In this case $\deg_T(v_3)=2$ or only paths $P_4$ are attached to $v_3.$ In this case, we subdivide the edges $v_0v_1, v_1v_2$, $v_2v_3$ and $v_3v_4$ with vertices $x_1, x_2, x_3,x_4$ respectively, obtaining the tree $T'.$ Let $D'$ be a minimum isolating set of $T'$ such that (by Observation \ref{o_1}) $\{v_1\} \subseteq D'.$  Observe that $|\{v_2,x_3,v_3\} \cap D'|\geq 1$ and without loss of generality let $v_3 \in D'.$ Let $T''$ be a component of $T'-x_4v_4$ containing $v_4$. Since $x_4$ is dominated by $v_3$ in $T',$ notice that $D' \cap V(T'')$ is a minimum isolating set of $T''.$ In this case $D' \setminus \{v_3\}$ is an isolating set of $T$ of cardinality $\iota(T')-1,$ so $\iota(T) \leq \iota(T')-1$ and we are done. 
\end{proof}

\subsection{Trees with isolation subdivision number equal to 1}
Now we are in position to characterize all trees $T$ with $\sdi(T)=1.$
Denote by  $\mathcal{N}_{\iota}(G)$ the set of those vertices which are not contained in any $\iota$-set of $G.$





Let $T=(V_T,E_T)$ be a tree and consider a vertex $r \in V_T$. Let $T_r$ denote the tree $T$ rooted at $r$, and let $C_r(v)$ denote the set of all children (in $T_r$) of a vertex $v \in V_T$. Finally, let $T_{r,x}$ denote the subtree in $T_r$, rooted at $x$, containing $x$ and all its descendants in $T_r$ (and so $T_{r,r}=T_r$). 

\begin{theorem}\label{theo:sd1} For any tree $T$ different from a star, $\sdi(T)=1$ if and only if there exists a vertex $v\in V_T$ such that $N_T[v] \subseteq \mathcal{N}_{\iota}(T).$
\end{theorem}
\begin{proof} ($\Leftarrow$)
Assume first that there exists a vertex $v$ such that $N_T[v] \subseteq \mathcal{N}_{\iota}(T)$. If $v$ is a leaf and $u$ is its support vertex, let $T'$ be the tree obtained by joining a new vertex $w$ to $v$ and consider any  $\iota$-set $D'$ of $T'$. Then  $|D'\cap \{u,v,w\}|=1$, and the set $D=(D'-\{u,v,w\})\cup \{u\}$ is an isolating set of $T$ containing $u$ such that $|D'|=|D|$. Since $N_T[v] \subseteq \mathcal{N}_{\iota}(T)$, we have $|D|>\iota (T)$. Hence, $\iota(T')=|D'|=|D|>\iota (T)$, and so $\sdi (T)=1$. Next, if $v$ is a support vertex, then the leaf $u$ attached to $v$ satisfies $N_T[u] \subseteq \mathcal{N}_{\iota}(T)$ and the former case applies. Hence, $\sdi(T)=1$. Finally, if $v$ is neither a leaf nor a support vertex of $T$,  then $v$ is the center of a star of order at least 3 such that no vertex of this star is a leaf of $T$.  Let $T'$ be the tree obtained by subdividing exactly one edge $uv$ of this star with a vertex $w$. Let $D'$ be an $\iota$-set of $T'$.
If $w\notin D'$, then at least one vertex from $N_T[v]$ is in $D'$ (otherwise, the edge $vw$ would not be ve-dominated by $D'$). Hence, $D'$ is also an isolating set of $T$ containing at least one vertex of $N_T[v]$. Thus, $\iota (T')=|D'|>\iota (T)$, implying that $\sdi (T)=1$.
If $w\in D'$ and exactly one of the vertices $u$ or $v$  belongs to $D'$, then  $D=(D'\setminus \{u,v,w\})\cup \{u,v\}$ is an isolating set of $T$ such that $\iota (T')=|D'|=|D|>\iota (T)$, because $N_T[v]\subseteq \mathcal{N}_{\iota}(T)$. Thus, $\sdi (T)=1$.
If $w\in D'$ and neither $u$ nor $v$  belong to $D'$, then  $D=(D'\setminus \{w\})\cup \{v\}$ is an isolating set of $T$
such that $\iota (T')=|D'|=|D|>\iota (T)$, because $N_T[v]\subseteq \mathcal{N}_{\iota}(T)$. Hence, $\sdi (T)=1$, which eventually finishes $(\Leftarrow)$. 

\medskip
\noindent $(\Rightarrow)$ Let $T=(V_T,E_T)$ be a tree with $\sdi (T)=1$ and assume that the subdivision of the edge $e=uv$ increases the isolation number of $T$. Let $T_e$ be the tree obtained from $T$ by subdividing the edge $e$ (and thus $\iota(T_e)>\iota (T)$). 

\medskip
\noindent \textbf{Case 1}: $e$  is a pendant edge; w.l.o.g.~assume that $u$ is a leaf. Let $D$ be an $\iota$-set of $T$. Then, $D$ does not contain both $u$ and $v$. Hence, if $u$ or $v$ belong to some $\iota$-set $D$ of $T$, then $D'=(D \setminus\{u,v\})\cup \{v\}$ is an $\iota$-set of $T_e$. Thus, $\iota (T_e)=|D_e|=|D|=\iota (T)$. Hence $u,v\in N_T[u]\subseteq \mathcal{N}_{\iota}(T)$. 
  
\medskip
\noindent \textbf{Case 2}. $e$  is not a pendant edge and there exists an $\iota$-set $D$ of $T$ such that either $u \in D$ or $v \in D$ (clearly, we cannot have both $u,v \in D$, because otherwise $\iota()T)=\iota(T_e)$, a contradiction); w.l.o.g.~assume that $v \in D$. It follows from the choice of  
$e$ and the above assumptions that there exist two edges $e_u,e_w$ in $E^2_T(v)$ such that:
\begin{itemize}
\item[$(a)$] $e_u$ and $e_w$ do not share a vertex;
\item[$(b)$] exactly one of them, say $e_u$, shares vertex $u$ with $e$;
\item[$(c)$] $e_u$ and $e_w$ are private edges for $v$ with respect to \textbf{any $\iota$-set $D^\ast$ of $T$ with $v \in D^\ast$}.
\end{itemize}
\noindent Consequently, assuming $e_u=uu'$ and $e_w=ww'$ (where $w$ is a neighbor of $v$), and still keeping in mind $\iota(T_e)>\iota(T)$, we have the following observation (see Figure~\ref{fig:Case2} for an illustration).

\begin{figure}[h]
\begin{center}

\begin{tikzpicture}[scale=0.8,
    thick,
    bnode/.style={circle, fill=black, draw=black, inner sep=1.8pt},
    wnode/.style={circle, fill=white, draw=black, inner sep=1.8pt, minimum size=0.25cm},
    triangle/.style={draw=black, thick, fill=white, join=round}
]

\draw[line width=12pt, lightgray, cap=round] 
    (-2, 1.8) .. controls (-3.1, 0) and (-2, -3.1) .. 
    (0, -3.1) .. controls (2, -3.1) and (3.1, 0) .. (2, 1.8);

\draw[line width=12pt, lightgray, cap=round] 
    (-0.9, 3.1) .. controls(-0.5,3.3) and (0.5,3.3).. (0.9, 3.1);

\node[font=\Large] at (-2, 2.3) {$^{C'}$};
\node[font=\Large] at (1.5, 3) {$^{C''}$};

\node[bnode, label=below:$_v$] (v) at (0, 2) {};
\node[wnode, label=right:$_u$] (u) at (-1.5, 0.5) {};
\node[wnode, label=left:$_w$]  (w) at (1.5, 0.5) {};
\node[wnode, label=right:$_{u'}$] (up) at (-1.5, -1) {};
\node[wnode, label=left:$_{w'}$]  (wp) at (1.5, -1) {};

\draw[line width=2pt] (v) -- (u) node[midway, above left, text=black] {$e$};
\draw (v) -- (w);
\draw[line width=2pt] (u) -- (up) node[midway, right, text=black] {$e_u$};
\draw[line width=2pt] (w) -- (wp) node[midway, left, text=black] {$e_w$};

\node[wnode] (vt1) at (-0.6, 3.2) {};
\node[wnode] (vt2) at (0.6, 3.2) {};
\draw (v) -- (vt1);
\draw (v) -- (vt2);
\node at (0.05, 2.7) {$^{\dots}$};

\node[wnode] (ul1) at (-2.4, 0.9) {};
\node[wnode] (ul2) at (-2.4, -0.1) {};
\draw (u) -- (ul1);
\draw (u) -- (ul2);
\node[rotate=90] at (-1.9, 0.5) {$^{\dots}$};

\node[wnode] (ub1) at (-1.9, -1.9) {};
\node[wnode] (ub2) at (-1.3, -2.6) {};
\draw (up) -- (ub1);
\draw (up) -- (ub2);
\node[rotate=-60] at (-1.6, -1.65) {$^{\dots}$};

\node[wnode] (wr1) at (2.4, 0.9) {};
\node[wnode] (wr2) at (2.4, -0.1) {};
\draw (w) -- (wr1);
\draw (w) -- (wr2);
\node[rotate=90] at (2, 0.5) {$^{\dots}$};

\node[wnode] (wb1) at (1.3, -2.6) {};
\node[wnode] (wb2) at (1.9, -1.9) {};
\draw (wp) -- (wb1);
\draw (wp) -- (wb2);
\node[rotate=60] at (1.6, -1.65) {$^{\dots}$};

\draw[triangle] (vt1) -- ++(0.1, 1.4) -- ++(-0.9, -0.6) -- cycle;
\draw[triangle] (vt2) -- ++(0.1, 1.4) -- ++(0.9, -0.6) -- cycle;

\draw[triangle] (ul1) -- ++(-0.8, 0.8) -- ++(-0.2, -1.2) -- cycle;
\draw[triangle] (ul2) -- ++(-0.9, 0.2) -- ++(0.2, -1.0) -- cycle;

\draw[triangle] (ub1) -- ++(-0.5, -1.2) -- ++(-0.8, 0.4) -- cycle;
\draw[triangle] (ub2) -- ++(-0.4, -1.2) -- ++(0.7, -0.2) -- cycle;

\draw[triangle] (wr1) -- ++(0.8, 0.8) -- ++(0.2, -1.2) -- cycle;
\draw[triangle] (wr2) -- ++(0.9, 0.2) -- ++(-0.2, -1.0) -- cycle;

\draw[triangle] (wb1) -- ++(0.4, -1.2) -- ++(-0.7, -0.2) -- cycle;
\draw[triangle] (wb2) -- ++(0.5, -1.2) -- ++(0.8, 0.4) -- cycle;

\end{tikzpicture}
\caption{Illustrating Case 2.}\label{fig:Case2}
\end{center}
\end{figure}
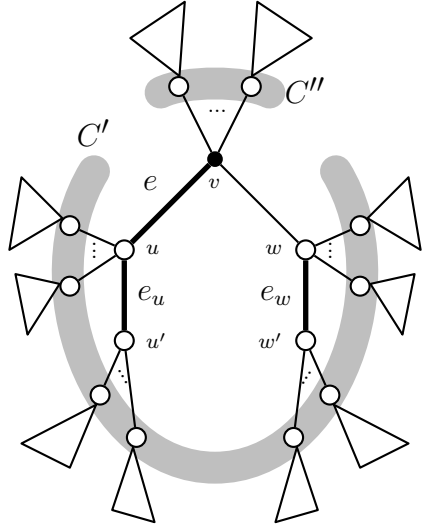
\begin{obs}\label{obs:u_uprim_w_wprim}
Consider the rooted tree $T_v$. Then:
\begin{itemize}
\item[$(a)$] Let $C'=(C_v(u) \setminus \{u'\}) \cup C_v(u')  \cup (C_v(w) \setminus \{w'\}) \cup C_v(w')$.  For each vertex $x \in C'$, any isolating ve-dominating set of $E_{T_{v,x}}$ in $T$ has at least $a_{x}$ elements in $V_{T_{v,x}}\setminus\{x\}$, where $a_{x}=|D \cap V_{T_{v,x}}|$.
\item[$(b)$] Let $C''=C_v(v) \setminus \{u,w\}$. For each vertex $y \in C''$, any isolating ve-dominating set of $E_{T_{v,y}}\setminus E^1(y)$ $E_{T_{v,y}}\setminus E^1_T(y)$ in $T$ has at least $a_{y}$ elements in $V_{T_{v,y}}$, where $a_{y}=|D \cap V_{T_{v,y}}|$.

\end{itemize}
\end{obs}

\noindent It follows from Observation~\ref{obs:u_uprim_w_wprim} that $$|D|=1+\sum_{x \in C'}a_x+\sum_{y \in C''}a_y.$$
Now, we claim that $N_T[u'] \subseteq \mathcal{N}_\iota(T)$. Indeed, suppose to the contrary that there exists $\hat{u} \in N_T[u']$ such that $\hat{u} \notin \mathcal{N}_\iota(T)$. Let $D_{\hat{u}}$ be an $\iota$-set of $T$ such that $\hat{u} \in D_{\hat{u}}$. Since ${\rm dist}_T(u,w)=2$, to isolate edge $e_w$, we must have $N[e_w] \cap D_{\hat{u}} \neq \emptyset$. But then, taking into account Observation~\ref{obs:u_uprim_w_wprim}, we obtain
$|D_{\hat{u}}| \ge 2+\sum_{x \in C'}{a_x}+\sum_{y \in C''}{a_y}>|D|,$ a contradiction.

\medskip
\noindent \textbf{Case 3}. $e$  is not a pendant edge and $u,v \in\mathcal{N}_\iota(T)$. The argument is similar to that in Case 2. Namely, let $D$ be an $\iota$-set of $T$. By the definition of $e$, we must have either $(N(u)\setminus \{v\}) \cap D \neq \emptyset$ or $(N(v)\setminus \{u\}) \cap D \neq \emptyset$. W.l.o.g.~assume that $q \in (N(v)\setminus \{u\}) \cap D$. Since $\iota(T_e)>\iota(T)$, there exists an edge $e_w$ in $E^2_T(q)$  such that:
\begin{itemize}
\item[$(a)$] $e$ and $e_w$ do not share a vertex;
\item[$(b)$] $e_w$ is a private edge for $q$ with respect to \textbf{any $\iota$-set $D^\ast$ of $T$ with $q \in D^\ast$}.
\end{itemize}
\noindent Consequently, assuming $e_w=ww'$ (where $w$ is a neighbor of $q$), still keeping in mind $\iota(T_e)>\iota(T)$, we have the following observation (see Figure~\ref{fig:Case3} for an illustration). 

\begin{figure}[h]
\begin{center}
\begin{tikzpicture}[scale=0.8,
    thick,
    bnode/.style={circle, fill=black, draw=black, inner sep=1.8pt},
    wnode/.style={circle, fill=white, draw=black, inner sep=1.8pt, minimum size=0.25cm},
    triangle/.style={draw=black, thick, fill=white, join=round}
]

\draw[line width=12pt, lightgray, cap=round] 
(-2.1, -1.5) .. controls (-2.1, -2) and (-1.5, -3) .. 
    (0, -3) .. controls (2, -3.1) and (3.1, 0) .. (2.2, 1.3);

\draw[line width=12pt, lightgray, cap=round] 
    (-1.0, 3.1) .. controls (-0.5, 3.3) and (0.5, 3.3) .. (1.0, 3.1);

\draw[line width=12pt, lightgray, cap=round] 
    (-2.35, 1.1) .. controls (-2.5, 0.5) .. (-2.4, -0.2);

\node at (-2, -1.1) {$^{C'}$};
\node at (1.6, 3.1) {$^{C''}$};
\node at (-2.0, 1.6) {$^{C'''}$};

\node[bnode, label=below:{$^q$}] (q) at (0, 2) {};
\node[wnode, label=right:{$_v$}] (v) at (-1.5, 0.5) {};
\node[wnode, label=right:{$_u$}] (u) at (-1.5, -1) {};
\node[wnode, label=left:{$_w$}]  (w) at (1.5, 0.5) {};
\node[wnode, label=left:{$_{w'}$}] (wp) at (1.5, -1) {};

\draw[line width=1pt] (q) -- (v);
\draw[line width=2pt] (v) -- (u) node[midway, right, text=black] {$^e$};
\draw[line width=1pt] (q) -- (w);
\draw[line width=2pt] (w) -- (wp) node[midway, left, text=black] {$^{e_w}$};

\node[wnode] (qt1) at (-0.6, 3.2) {};
\node[wnode] (qt2) at (0.6, 3.2) {};
\draw (q) -- (qt1);
\draw (q) -- (qt2);
\node at (0.05, 2.7) {$^{\dots}$};

\node[wnode] (vl1) at (-2.4, 0.9) {};
\node[wnode] (vl2) at (-2.4, -0.1) {};
\draw (v) -- (vl1);
\draw (v) -- (vl2);
\node[rotate=90] at (-1.9, 0.5) {$^{\dots}$};

\node[wnode] (ul1) at (-1.9, -1.9) {};
\node[wnode] (ul2) at (-1.3, -2.6) {};
\draw (u) -- (ul1);
\draw (u) -- (ul2);
\node[rotate=-60] at (-1.6, -1.65) {$^{\dots}$};

\node[wnode] (wr1) at (2.4, 0.9) {};
\node[wnode] (wr2) at (2.4, -0.1) {};
\draw (w) -- (wr1);
\draw (w) -- (wr2);
\node[rotate=90] at (2, 0.5) {$^{\dots}$};

\node[wnode] (wb1) at (1.3, -2.6) {};
\node[wnode] (wb2) at (1.9, -1.9) {};
\draw (wp) -- (wb1);
\draw (wp) -- (wb2);
\node[rotate=60] at (1.6, -1.65) {$^{\dots}$};

\draw[triangle] (qt1) -- ++(0.1, 1.4) -- ++(-0.9, -0.6) -- cycle;
\draw[triangle] (qt2) -- ++(0.1, 1.4) -- ++(0.9, -0.6) -- cycle;

\draw[triangle] (vl1) -- ++(-0.8, 0.8) -- ++(-0.2, -1.2) -- cycle;
\draw[triangle] (vl2) -- ++(-0.9, 0.2) -- ++(0.2, -1.0) -- cycle;

\draw[triangle] (ul1) -- ++(-0.5, -1.2) -- ++(-0.8, 0.4) -- cycle;
\draw[triangle] (ul2) -- ++(-0.4, -1.2) -- ++(0.7, -0.2) -- cycle;

\draw[triangle] (wr1) -- ++(0.8, 0.8) -- ++(0.2, -1.2) -- cycle;
\draw[triangle] (wr2) -- ++(0.9, 0.2) -- ++(-0.2, -1.0) -- cycle;

\draw[triangle] (wb1) -- ++(0.4, -1.2) -- ++(-0.7, -0.2) -- cycle;
\draw[triangle] (wb2) -- ++(0.5, -1.2) -- ++(0.8, 0.4) -- cycle;

\end{tikzpicture}
\caption{Illustrating Case 3.}\label{fig:Case3}
\end{center}
\end{figure}
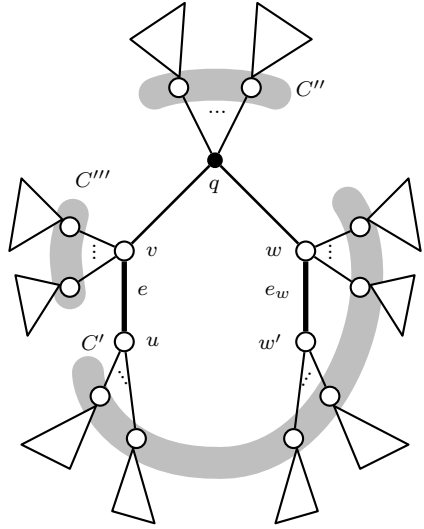

\begin{obs}\label{obs:v_u_w_wprim}
Consider the rooted tree $T_q$. Then:
\begin{itemize}
\item[$(a)$] Let $C'=C_q(u)  \cup (C_q(w) \setminus \{w'\}) \cup C_q(w')$. For each vertex $x \in C'$, any isolating ve-dominating set of $E_{T_{q,x}}$ in $T$ has at least $b_{x}$ elements in $V_{T_{q,x}}\setminus\{x\}$, where $b_{x}=|D \cap V_{T_{q,x}}|$.
\item[$(b)$] Let $C''=C_q(q) \setminus \{v,w\}$. For each vertex $y \in C''$, any isolating ve-dominating set of $E_{T_{q,y}}\setminus E^1_T(y)$ in $T$ has at least $b_{y}$ elements in $V_{T_{q,y}}$, where $b_{y}=|D \cap V_{T_{q,y}}|$.
\item[$(c)$] Let $C'''=C_q(v) \setminus \{u\}$. For each vertex $z \in C'''$, any isolating set of $E_{T_{q,z}}$ in $T$ has at least $b_{z}$ elements in $V_{T_{q,z}}$, where $b_{z}=|D \cap V_{T_{q,z}}|$.
\end{itemize}
\end{obs}

\noindent It follows from Observation~\ref{obs:v_u_w_wprim} that $$|D|=1+\sum_{x \in C'}b_x+\sum_{y \in C''}b_y+\sum_{z \in C'''}b_z.$$

Now, similarly as in Case 2, we claim that $N_T[u] \subseteq \mathcal{N}_\iota(T)$. Indeed, suppose to the contrary that there exists $\hat{u} \in N_T[u]$ such that $\hat{u}  \notin \mathcal{N}_\iota(T)$. Let $D_{\hat{u} }$ be an $\iota$-set of $T$ such that $\hat{u}  \in D_{\hat{u}}$. Since $d_T(v,w)=2$, to isolate edge $e_w$, we must have $N[e_w] \cap D_{\hat{u} } \neq \emptyset$. But then, taking into account Observation~\ref{obs:v_u_w_wprim}, we obtain
$|D_{\hat{u}}| \ge 2+\sum_{x \in C'}{b_x}+\sum_{y \in C''}{b_y}+\sum_{z \in C'''}{b_z}>|D|,$ a contradiction.
\end{proof}

\section{Open problems}
Clearly, we are left with the problem of characterizing the classes of trees $T$ with isolation subdivision number $sd_{\iota}(T)=2,3$ and $4$, respectively. Bearing in mind Theorem~\ref{theo:sd1}, one could expect that the class of trees $T$ with $\sdi(T)=4$ is equivalent to the class of trees $T$ with $\mathcal{N}_\iota(T)=\emptyset$, however, Figure~\ref{fig:Class_4}
 shows that such a conjecture is false.

\begin{figure}[h]
\begin{center}
\includegraphics[width=0.15\textwidth]{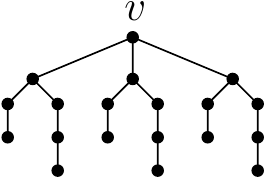}
\caption{A tree $T$ with $\sdi(T)=4$ and $\mathcal{N}_\iota(T)=\{v\} \neq \emptyset$.}\label{fig:Class_4}
\end{center}
\end{figure}

\paragraph{Algorithmic aspects.} In 2025, Chen et al.~\cite{ChLWX25} gave a linear time algorithm for determining a minimum isolating set of a tree (among other interesting algorithmic results).  A natural question is how to check algorithmically (and efficiently) whether a given tree $T$ has $sd_{\iota}(T)=1$. According to Theorem~\ref{theo:sd1}, $sd_{\iota}(T)=1$ if and only if there is at least one vertex $v\in T$ such that $N_T[v] \subseteq \mathcal{N}_{\iota}(T)$.  Recall that there are several techniques for determining the domination number in trees in linear time~\cite{Ch13}, and one can also argue that the relevant standard dynamic programming-based approach, combined with rooting technique, allows us to compute whether a given vertex $v$ belongs $\mathcal{N}_\iota(T)$ to in linear time either, and therefore, the set $\mathcal{N}_\iota(T)$ can be computed in quadratic time. Then, by  traversing a given tree in a DFS-like manner, we can easily verify in linear time whether there is at least one vertex $v\in T$ such that $N_T[v] \subseteq \mathcal{N}_{\iota}(T)$ (and so whether $sd_{\iota}(T)=1$), which results in total quadratic time either.

Nevertheless, we expect that by exploiting a common generalized re-rooting technique, as it was done in~\cite{ZZ25} for a few domination-like problems, given a tree $T$, one can compute the set $\mathcal{N}_\iota(T)$ in linear time, which immediately results in linear time algorithm for verifying whether $sd_{\iota}(T)=1$. However, since the formal description looks to be very technical and tedious, but not inspiring, we decided to leave the formal proof for a future research.

\end{document}